\documentclass[12pt, reqno]{amsart}
\usepackage{amsmath, amsthm, amscd, amsfonts, amssymb, graphicx, xcolor}
\usepackage[bookmarksnumbered, colorlinks, plainpages]{hyperref}

\newtheorem{theorem}{Theorem}[section]
\newtheorem{lemma}[theorem]{Lemma}
\newtheorem{proposition}[theorem]{Proposition}

\theoremstyle{definition}

\newtheorem{example}[theorem]{Example}

\newtheorem{question}[theorem]{Question}
\theoremstyle{remark}
\newtheorem{remark}[theorem]{Remark}
\numberwithin{equation}{section}

\newcommand{\hyper}{\mathcal{H}}
\newcommand{\finite}{\mathcal{F}}

\begin{document}
\color{darkgray}

\title[A note on the structural stability of almost one-to-one maps]{A note on the structural stability of almost one-to-one maps}

\author[]{Mar\'{\i}a Isabel Cortez}
\address{Facultad de Matem\'aticas, Pontificia Universidad Cat\'olica de Chile. Edificio Rolando Chuaqui, Campus San Joaquín. Avda. Vicuña Mackenna 4860, Macul, Chile.}
\email{maria.cortez@uc.cl}

\author[]{Till Hauser}
\address{Facultad de Matem\'aticas, Pontificia Universidad Cat\'olica de Chile. Edificio Rolando Chuaqui, Campus San Joaquín. Avda. Vicuña Mackenna 4860, Macul, Chile.}
\email{hauser.math@mail.de}

\thanks{This article was funded by the Deutsche Forschungsgemeinschaft (DFG, German Research Foundation) – 530703788.}

\begin{abstract}
     A continuous surjection  $\pi:X\to Y$   between  compact Hausdorff spaces induces continuous surjections $\mathcal{M}(\pi)\colon \mathcal{M}(X)\to\mathcal{M}(Y)$ and $\hyper(\pi): \hyper(X)\to\hyper(Y)$ between the spaces of regular Borel probability measures, and the spaces of closed subsets, respetively. 
     It is well known that $\hyper(\pi)$ is irreducible if and only if $\pi$ is irreducible.
     We show that $\mathcal{M}(\pi)$ is irreducible if and only if $\pi$ is irreducible. 
     Furthermore, we show that whenever $\pi$ is almost one-to-one then $\mathcal{M}(\pi)$ and $\hyper(\pi)$ are almost one-to-one.  In particular, we observe that continuous surjections between compact metric spaces are almost one-to-one if and only if $\hyper(\pi)$ is almost one-to-one and a similar statement about $\mathcal{M}(\pi)$. Finally, we give alternative proofs for some results in \cite{DX24} regarding semi-open maps. 
\end{abstract} \maketitle

\section{Introduction}

Almost one-to-one maps (which we will define below) are a natural ingredient to the study and characterization of topological dynamical systems  \cite{Au88, FW89, Dow05, W12}. Related properties are openess, semi-openess and irreducibility, which all naturally appear in the study of topological dynamics \cite{GW77, Au88, Gl07} and which we will discuss in more detail below. 
Since several interesting examples in topological dynamics arise from dynamical systems with non-metric phase spaces \cite{ellis1969lectures, Au88} we consider it to be natural examine these notions as purely topological notions for continuous surjections between compact Hausdorff spaces. 

A continuous and surjective mapping $\pi\colon X\to Y$ between compact Hausdorff spaces is called \emph{open}, whenever images of open sets are open. It is called \emph{semi-open}, whenever images of non-empty open sets have a non-empty interior. It is shown in \cite[Lemma 2.1]{Gl07} and in \cite[Lemma 2]{DX24} that $\pi$ is semi-open, if and only if the preimage of any dense set is dense. 

A point $x\in X$ is called an \emph{injectivity point}, if and only if $\pi^{-1}(\pi(x))=\{x\}$, i.e.\ if $x$ is the only point mapping to $\pi(x)$. 
A continuous and surjective map $\pi\colon X\to Y$ is called \emph{almost one-to-one}, if and only if the set $X_0$ of injectivity points is dense in $X$ \cite[Page 157]{Au88}. 
A continuous and surjective mapping $\pi\colon X\to Y$ between compact Hausdorff spaces is called \emph{irreducible}, whenever for all closed set $A\subseteq X$ with $\pi(A)=Y$ we have that $A=X$. Note that this property is also called highly proximality in the context of factor maps between minimal actions on compact Hausdorff spaces \cite{GW77}.
Furthermore, $\pi$ is irreducible, if and only if for every nonempty open $U\subseteq X$ there exists $y\in Y$ with $\pi^{-1}(y)\subseteq U$, and if and only if for every nonempty open $U\subseteq X$ there exists a nonempty and open subset $V\subseteq Y$ with $\pi^{-1}(V)\subseteq U$ \cite{AG01}. 
It follows that any almost one-to-one map is irreducible and that any irreducible map is semi-open. If $X$ is metric, then the notions of irreducibility and almost one-to-one maps coincide \cite[Lemma 1.1]{AG01}, but there exist irreducible maps between compact Hausdorff spaces that are not almost one-to-one \cite[Example 5.29]{ellis1969lectures}. 

Note that some authors prefer to define almost one-to-one maps by asking for $\pi(X_0)$ to be dense in $Y$, where $X_0$ is the set of injectivity points of $\pi$. 
Considering $\pi\colon [-1,1]\to [0,1]$ which maps $[-1,0]$ to $0$ and acts as the identity on $[0,1]$ we note that this property is strictly weaker than our definition. 
Combining the previous comments, we observe that $\pi$ is almost one-to-one, if and only if it is semi-open and $\pi(X_0)$ is dense in $Y$. Since any factor map between minimal actions is semi-open, we see that the two versions of the definition agree in the minimal context. 

Another important object in the study of topological dynamics is the hyperpace of a compact Hausdorff space $X$ \cite{BS75, GW95, Gla00, KO07, Nag22}.
Denote $\hyper(X)$ for the \emph{hyperspace of $X$}, i.e.\ the set of all closed and non-empty subsets $A\subseteq X$. 
It is equipped with the \emph{Vietoris topology}, for which a base is given by the sets of the form 
\[\langle U_1,\dots,U_n\rangle:=\left\{A\in \hyper(X);\, A\subseteq \bigcup_{i=1}^n U_i, \text{ and } A\cap U_i\neq \emptyset\text{ for all } i=1,\dots,n\right\}\]
with $n\geq 1$ and open and non-empty subsets $U_1,\dots,U_n$ of $X$. 
Recall from \cite[Proposition 4.1, Theorem 4.9]{E51} that $\hyper(X)$ is a compact Hausdorff space and that it is metrizable, if and only if $X$ is metrizable.

Equally important as the hyperspace is the set of all regular Borel probability measures on $X$ equipped with the weak*-topology, which we will denote by $\mathcal{M}(X)$ \cite{BS75, GW95}.  
It is well-known that from the Banach-Alaoglu theorem it follows that $\mathcal{M}(X)$ is a compact Hausdorff space and it is straightforward to show that it is metrizable, if and only if $X$ is metrizable \cite{DGS76}. 

Now consider a continuous surjection $\pi\colon X\to Y$. 
We denote $\hyper(\pi)$ for the mapping $\hyper(X)\ni A\mapsto \pi(A)\in \hyper(Y)$. Note that $\hyper(\pi)$ is a continuous surjection between the hyperspaces. 
Furthermore, the push forward mapping $\mathcal{M}(X)\ni \mu\mapsto (f\mapsto \mu(f\circ \pi))\in \mathcal{M}(Y)$ will be denoted by $\mathcal{M}(\pi)$. 
It is a continuous and affine surjection. 
Note that $\hyper$ and $\mathcal{M}$ are functorial, i.e.\ behave well under composition. 
It is natural to study the stability of properties of continuous surjections under $\hyper$ and $\mathcal{M}$. We next collect some results from the literature.  

\pagebreak

\begin{theorem}
Let $\pi$ be a cont.\ surjection between compact Hausdorff spaces. 
\begin{itemize}
    \item[(${\mathcal{M}}_o$)] $\mathcal{M}(\pi)$ is open, if and only if $\pi$ is open \cite[§4]{DE72}. 
    
    \item[(${\hyper}_o$)] $\hyper(\pi)$  is open, if and only if $\pi$ is open \cite[Thm.\ 4.3]{Hos97}.
    
    \item[(${\mathcal{M}}^{\Rightarrow}_s$)] If $\mathcal{M}(\pi)$ is semi-open, then $\pi$ is semi-open \cite[Thm.\ B'']{DX24}.  
    
    \item[(${\mathcal{M}}^{m}_s$)] If $\pi$ is a continuous surjection between compact metrizable spaces, then $\mathcal{M}(\pi)$ is semi-open, if and only if $\pi$ is semi-open \cite[Thm.\ 2.3]{Gl07} and \cite[Thm.\ B'']{DX24}. 
    
    \item[($\hyper_s$)]  $\hyper(\pi)$ is semi-open, if and only if $\pi$ is semi-open \cite[Thm.\ 4]{DX24}. 

    \item[($\hyper_i$)] $\hyper(\pi)$ is irreducible, if and only if $\pi$ is irreducible 
    \cite[Thm.\ 9B]{DX24}.
\end{itemize}
In particular, we observe the following from ($\hyper_i$): 
\begin{itemize}
    \item[($\hyper^{m}_{a}$)] If $\pi$ is a continuous surjection between compact metrizable spaces, then $\hyper(\pi)$ is almost one-to-one, if and only if $\pi$ is almost one-to-one. 
\end{itemize}
\end{theorem}

In this article we will add to this list of results by showing the following in Section \ref{sec:proofMainThm}. 

\begin{theorem}
    \label{thm:mainResults}
    Let $\pi$ be a cont.\ surjection between compact Hausdorff spaces. 
\begin{itemize}
    \item[($\mathcal{M}_i$)] $\mathcal{M}(\pi)$ is irreducible, if and only if $\pi$ is irreducible.
    \item[($\mathcal{M}^{\Leftarrow}_a$)] If $\pi$ is almost one-to-one, then $\mathcal{M}(\pi)$ is almost one-to-one.
     \item[($\hyper^{\Leftarrow}_a$)] If $\pi$ is almost one-to-one, then $\hyper(\pi)$ is almost one-to-one. 
\end{itemize}
In particular, we observe the following from ($\mathcal{M}_i$): 
\begin{itemize}
    \item[($\mathcal{M}^{m}_a$)] If $\pi$ is a continuous surjection between compact metrizable spaces, then $\mathcal{M}(\pi)$ is almost one-to-one, if and only if $\pi$ is almost one-to-one.
\end{itemize}
\end{theorem}

In addition, Section \ref{sec:alternativeProof} presents an argument that allows for alternative and more direct proofs of ($\mathcal{M}^{\Rightarrow}_s$) and the implication $\Rightarrow$ in ($\hyper_s$).
The following natural questions remain open. Note that they are all answered affirmatively if $X$ is a compact metric space. 
See \cite{DX24} for a partial answer on the first question. 

\begin{question}
Let $\pi$ be a cont.\ surjection between compact Hausdorff spaces.  
\begin{itemize}
\item[($\mathcal{M}^{\Leftarrow}_s$)] If $\pi$ is semi-open, 
is $\mathcal{M}(\pi)$ semi-open?  
\item[($\mathcal{M}^{\Rightarrow}_a$)] If $\mathcal{M}(\pi)$ is almost one-to-one, 
is $\pi$ almost one-to-one?
\item[($\hyper^{\Rightarrow}_a$)] If $\hyper(\pi)$ is almost one-to-one, is $\pi$ is almost one-to-one?
\end{itemize}
\end{question}

In Section \ref{sec:productsDecompositionComposition}, we present further stability properties of the discussed notions, regarding product, composition and decomposition of maps.

\section{Irreducible and almost one-to-one maps}
\label{sec:proofMainThm}
In this section we prove various propositions, which combined yield the statement of Theorem \ref{thm:mainResults}. 
Consider a compact Hausdorff space $X$. 
We denote $\delta^{(X)}\colon X\to \mathcal{M}(X)$ for the mapping that maps $x$ to the Dirac measure $\delta_x$.  With a slight abuse of notation, we simply write $\delta$, whenever the space is clear from the context. 

\begin{proposition}
\label{pro:IM}
    Let $\pi\colon X\to Y$ be a continuous surjection between compact Hausdorff spaces. 
\begin{itemize}
    \item[($\mathcal{M}_i$)] $\mathcal{M}(\pi)$ is irreducible, if and only if $\pi$ is irreducible.
\end{itemize}
\end{proposition}

We will need to use Milman's theorem \cite[Proposition 1.5]{Phe01}. 

\begin{lemma}[Milman]
\label{lem:milman}
    Suppose that $M$ is a compact convex subset of a locally convex space, that $A \subseteq M$, and that $M$ is the closed convex hull of $A$. Then the extreme points of $M$ are contained in the closure of $A$.
\end{lemma}
\begin{proof}[Proof of Proposition \ref{pro:IM}:]

    '$\Rightarrow$': Let $A\subseteq X$ be closed, such that $\pi(A)=Y$. 
    Denote $M$ for the closed convex hull of $\delta(A)$. 
    We observe $\mathcal{M}(\pi)(\delta(A))=\delta(\pi(A))=\delta(Y)$. 
    Since $\mathcal{M}(\pi)$ is affine, continuous and closed we have that $\mathcal{M}(\pi)(M)$ is the closed convex hull of $\mathcal{M}(\pi)(\delta(A))=\delta(Y)$ and hence $\mathcal{M}(\pi)(M)=\mathcal{M}(Y)$. 
    Now recall that $\mathcal{M}(\pi)$ is assumed to be irreducible, which implies $M=\mathcal{M}(X)$. 

    We have shown that $\mathcal{M}(X)$ is the closed convex hull of $\delta(A)$. Since $\delta$ is a ho\-meo\-mor\-phism onto it's image we observe that $\delta(A)$ is closed and hence Milman's theorem (see Lemma \ref{lem:milman}) yields that the extreme points of $\mathcal{M}(X)$ are contained in $\delta(A)$, i.e.\ $\delta(X)\subseteq \delta(A)$. This implies $X=A$.

    '$\Leftarrow$': Consider a closed subset $M\subseteq \mathcal{M}(X)$ with $\mathcal{M}(\pi)(M)=\mathcal{M}(Y)$. 
    In order to show that $M=\mathcal{M}(X)$ we will show that $M$ contains all finite convex combinations of elements of $\delta(X)$. Since the set of all such finite convex combinations is dense in $\mathcal{M}(X)$ the statement then follows. 
    
    Consider such a finite convex combination $\mu=\sum_{i=1}^n \alpha_i \delta_{x_i}$ of elements of $\delta(X)$. 
    Since $M$ is closed and $\mathcal{M}(X)$ is equipped with the weak*-topology for $\epsilon>0$ and a finite subset $F\subseteq C(X)$ we need to find $\nu\in M$ with $|\mu(f)-\nu(f)|\leq \epsilon$ for all $f\in F$ in order to observe that $\mu\in M$. 
    
    Let $\epsilon>0$ and $F\subseteq C(X)$ finite. 
    Let $U_i:=\bigcap_{f\in F}\{x\in X;\, |f(x)-f(x_i)|<\epsilon\}$ and note that $U_i$ is an open neighborhood of $x_i$. 
    Since $\pi$ is irreducible there exists $y_i\in Y$ with $\pi^{-1}(y_i)\subseteq U_i$. 
    From $\sum_{i=1}^n \alpha_i \delta_{y_i}\in \mathcal{M}(Y)$ we observe that there exists $\nu\in M$ with $\mathcal{M}(\pi)(\nu)=\sum_{i=1}^n \alpha_i \delta_{y_i}$. Clearly, we can pick $\nu_i\in \mathcal{M}(\pi^{-1}(y_i))$, such that $\nu=\sum_{i=1}^n \alpha_i \nu_i$.
    Note that we have $\nu_i(U_i)=1$, which allows to compute the following for $f\in F$ and all $i$
\begin{align*}
    \left|\delta_{x_i}(f)-\nu_i(f)\right|
    = \left|\int_{U_i} f(x_i)-f(x) d\nu_i(x)\right|
    \leq \int_{U_i} \left|f(x_i)-f(x)\right| d\nu_i(x)
    \leq \epsilon.
\end{align*}
This shows
$\left|\mu(f)-\nu(f)\right|
\leq \sum_{i=1}^n \alpha_i \left|\delta_{x_i}(f)-\nu_i(f)\right|
\leq \epsilon$ for all $f\in F$.
\end{proof}

\begin{proposition}
\label{pro:AM}
    Let $\pi\colon X\to Y$ be a continuous surjection between compact Hausdorff spaces. 
\begin{itemize}
    \item[($\mathcal{M}^\Leftarrow_a$)] If $\pi$ is almost one-to-one, then $\mathcal{M}(\pi)$ is almost one-to-one.
\end{itemize}
\end{proposition}
\begin{proof}
    Denote $X_0$ for the set of injectivity points in $X$. 
    Since $X_0$ is dense in $X$ we observe that $\delta(X_0)$ is dense in $\delta(X)$ and hence that the convex hull $c\delta(X_0)$ of $\delta(X_0)$ is dense in $\mathcal{M}(X)$. 
    It thus suffices to show that $c\delta(X_0)$ consists of injectivity points of $\mathcal{M}(\pi)$. 
    For this consider $\mu\in c\delta(X_0)$ and $\nu\in \mathcal{M}(X)$ with $\mathcal{M}(\pi)\mu=\mathcal{M}(\pi)\nu$. 
    Since $\mu$ has a finite support in $X_0$ we observe that $\mathcal{M}(\pi)\mu$ has a finite support in $Y_0:=\pi(X_0)$. 
    Thus, since $X_0$ consists of injectivity points, also $\nu$ must have a support in $X_0$ and a moments thought reveals that $\mu=\nu$.    
\end{proof}

\begin{proposition}
\label{pro:AH}
    Let $\pi\colon X\to Y$ be a continuous surjection between compact Hausdorff spaces. 
\begin{itemize}
    \item[($\hyper^\Leftarrow_a$)] If $\pi$ is almost one-to-one, then $\hyper(\pi)$ is almost one-to-one.
\end{itemize}
\end{proposition}
\begin{proof}
     Let $X_0$ be the set of injectivity points of $\pi$ in $X$ and denote $\finite(X_0)$ for the set of finite subsets of $X_0$. 
    Since $X_0$ is dense in $X$ we observe $\finite(X_0)$ to be dense in $\hyper(X)$. It thus suffices to show that $\finite(X_0)$ consists of injectivity points w.r.t.\ $\hyper(\pi)$. 

    For this consider $F\in \finite(X_0)$ and $A\in \hyper(X)$ with $ \hyper(\pi)(F)=\hyper(\pi)(A)$, i.e.\ $\pi(F)=\pi(A)$. 
    For $x\in A$ we observe $\pi(x)\in \pi(A)=\pi(F)\subseteq \pi(X_0)$, from which we see that $x$ is an injectivity point of $\pi$. This shows $A\subseteq X_0$. 
    Since the restriction of $\pi$ to a map $X_0\to \pi(X_0)$ is bijective, we deduce $F=A$ from $\pi(F)=\pi(A)$. This shows $F$ to be an injectivity point of $\hyper(\pi)$. 
\end{proof}

\section{Semi-open maps}
\label{sec:alternativeProof}

    We will next present an alternative and less technical proof for \cite[Theorem B'']{DX24} ($\mathcal{M}^{\Rightarrow}_s$) and one of the implications in \cite[Thm.\ 4]{DX24} ($\Rightarrow$ in $\hyper_s$) using that a continuous surjection is semi-open, if and only if all preimages of dense sets are dense.

    \begin{proposition}
        Let $\pi$ be a cont.\ surjection between compact Hausdorff spaces. 
    \begin{itemize}
        \item[($\mathcal{M}^{\Rightarrow}_s$)] If $\mathcal{M}(\pi)$ is semi-open, then $\pi$ is semi-open. 
        
        \item[($\mathcal{H}^{\Rightarrow}_s$)] If $\hyper(\pi)$ is semi-open, then $\pi$ is semi-open.
    \end{itemize}
    \end{proposition}
    \begin{proof}
    '($\mathcal{M}^{\Rightarrow}_s$)': Let $D\subseteq Y$ be dense and note that the convex hull $c\delta(D)$ of $\delta(D)$ is dense in $\mathcal{M}(Y)$. 
    From the semi-openness of $\mathcal{M}(\pi)$ we thus observe that 
    $\mathcal{M}(\pi)^{-1}(c\delta(D))$ is dense in $\mathcal{M}(X)$. 
    Any $\mu\in \mathcal{M}(\pi)^{-1}(c\delta(D))$ is supported on $X':=\overline{\pi^{-1}(D)}$, from which we observe $\mathcal{M}(\pi)^{-1}(c\delta(D))\subseteq \mathcal{M}(X')$. 
    Since the latter is compact and hence a closed subset of $\mathcal{M}(X)$ we observe $\mathcal{M}(X')=\mathcal{M}(X)$, i.e.\ $X=X'$. 

    '($\mathcal{H}^{\Rightarrow}_s$)': Let $D\subseteq Y$ be dense and consider the set $\finite(D)$ of all finite subsets of $D$. 
    Since $\finite(D)$ is dense in  $\hyper(Y)$ we observe from the semi-openess of $\pi$ that $\hyper(\pi)^{-1}(\finite(D))$ is dense in $\hyper(X)$. Denote $X':=\overline{\pi^{-1}(D)}$ and note that $\hyper(\pi)^{-1}(\finite(D))\subseteq \hyper(X')$. Since the latter is compact and hence a closed subset of $\hyper(X)$ we observe $\hyper(X)=\hyper(X')$, i.e.\ $X=X'$. 
    \end{proof}

\section{About products, composition and decomposition}
\label{sec:productsDecompositionComposition}

The proof of the following proposition is straightforward. We include the proofs for almost one-to-one and irreducible maps for the convenience of the reader. 

\begin{proposition}
    Let $(\pi_i\colon X_i\to Y_i)_{i\in I}$ be a family of continuous surjections and denote $\pi:=\prod_{i\in I} \pi_i$. 
    Then $\pi$ is almost one-to-one, if and only if all $\pi_i$ are almost one-to-one. A similar statement holds about irreducible, open and semi-open maps. 
\end{proposition}
\begin{proof} 
    It is straight forward to show that the injectivity points of $\pi$ are given by $\times_{i\in I}X_{i,0}$. 
    The statement about almost one-to-one factor maps now follows, since $\times_{i\in I}X_{i,0}$ is dense in $\times_{i\in I}X_{i}$, if and only if for all $i$ we have that  $X_{i,0}$ is dense in $X_i$. 

    We next show that whenever $\pi$ is irreducible, then all the $\pi_i$ are irreducible.  
    Let $j\in I$ and consider $V\subseteq X_j$ open and non-empty. 
    Let $U_i:=X$ for $i\neq j$ and $U_j:=V$. Consider $U:=\times_i U_i$. Then $U$ is a non-empty and open subset of $\times_i X_i$. 
	By the irreducibility of $\pi$ there exists $y=(y_i)\in \times_i Y_i$ with $\times_i \pi_i^{-1}(y_i)=\pi^{-1}(y)\subseteq U$. In particular, we observe that $\pi_j^{-1}(y_j)\subseteq V$. This shows $\pi_j$ to be irreducible.

    For the converse assume that all $\pi_i$ are irreducible. 
    Let $U\subseteq \times_i X_i$ be open and non-empty. 
	There exist $(U_i)_{i\in I}$, such that $U_i$ is an open non-empty subset of $X_i$ for all $i\in I$ with $\times_i U_i\subseteq U$ (and all but finitely many satisfy $U_i=X_i$, we will not need the latter). 
	Now for $U_i$ choose $y_i\in Y_i$ with $\pi^{-1}(y_i)\subseteq U_i$. Denote $y:=(y_i)$ and note that $y\in \times_i Y_i$. 
	Clearly $\pi^{-1}(y)=\times_i \pi_i^{-1}(y_i)\subseteq \times_i U_i \subseteq U$. This shows that $\pi$ is irreducible. 
\end{proof}

In the following proposition we collect the composition and decomposition properties of the discussed notions. As above we include the details about almost one-to-one and irreducible factor maps for the convenience of the reader, but leave the proofs about open and semi-open factor maps to the reader.      

\begin{proposition}
    Let $\phi\colon X\to Y$ and $\psi\colon Y\to Z$ be continuous surjections between compact Hausdorff spaces and denote $\pi:=\psi\circ \phi$. 
\begin{itemize}
    \item[(i)] $\pi$ is irreducible, if and only if $\phi$ and $\psi$ are irreducile. 
    \item[(a)] If $\pi$ is almost one-to-one, then $\phi$ and $\psi$ are almost one-to-one. If $X$ is metrizable, then $\pi$ is almost one-to-one, if and only if $\phi$ and $\psi$ are almost one-to-one. 
    \item[(o)] If $\phi$ and $\psi$ are open, then $\pi$ is open. If $\pi$ is open, then $\psi$ is open. 
    \item[(s)] If $\phi$ and $\psi$ are semi-open, then $\pi$ is semi-open. If $\pi$ is semi-open, then $\psi$ is semi-open. 
\end{itemize}
\end{proposition}
\begin{proof}
'(i)': Assume that $\pi$ is irreducible. 
Let $A\subseteq X$ be a closed subset with $\phi(A)=Y$. Then $\pi(A)=Z$ and the irreducibility of $\pi$ yields $A=X$. This shows $\phi$ to be irreducible. 
Let $A\subseteq Y$ be a closed subset with $\psi(A)=Z$. Then $\phi^{-1}(A)$ is a closed subset of $X$ with $\pi(\phi^{-1}(A))=Z$ and the irreducibility of $\pi$ yields $\phi^{-1}(A)=X$, from which we observe $A=Y$. 

Now assume that $\phi$ and $\psi$ are irreducible. Let $A\subseteq X$ be a closed subset with $\pi(A)=Z$. Now $\phi(A)$ is a closed subset of $Y$ wtih $\psi(\phi(A))=Z$ and the irreducibility of $\psi$ yields $\phi(A)=Y$. From the irreducibility of $\phi$ we observe $A=X$. 

'(a)': Assume that $\pi$ is almost one-to-one. 
Since any injectivity point of $\pi$ is an injectivity point of $\phi$ we observe that $\phi$ is also almost one-to-one. 
If $X_0$ is the set of injectivity points of $\pi$, then $\phi(X_0)$ is contained in the set of injectivity points of $\psi$. 
Since $X_0$ is dense and $\phi$ is a continuous surjection we observe that also $\phi(X_0)$ is dense in $Y$. This shows $\psi$ to be almost one-to-one. 
The case of a metrizable $X$ follows from (i).     
\end{proof}

The following example shows that the (semi-)openess of $\phi$ can not be deduced from the (semi-)openess of $\pi$.  

\begin{example}
    Consider $X:=[-1,1]$ and $Y:=[0,1]$ and the factor map $\phi\colon X\to Y$ that sends $[-1,0]$ to $0$ and acts as the identity on $[0,1]$. Clearly $\phi$ is not semi-open. Furthermore, we consider  $Z:=\{0\}$ and the (unique) maps $\pi\colon X\to Z$ and $\psi\colon Y\to Z$. Note that $\pi:=\psi\circ \phi$ and $\psi$ are open. 
\end{example}

\begin{remark}
    Recall that if we consider continuous surjections between compact metric spaces, then the composition of almost one-to-one factor maps is also almost one-to-one. It remains open, whether this is also true in the context of continuous surjections between compact Hausdorff spaces. 
\end{remark}

\bibliographystyle{alpha}
\bibliography{ref}

\begin{thebibliography}{DGS76}

\bibitem[AG77]{GW77}
J.~Auslander and S.~Glasner.
\newblock Distal and highly proximal extensions of minimal flows.
\newblock {\em Indiana Univ. Math. J.}, 26(4):731--749, 1977.

\bibitem[AG01]{AG01}
Ethan Akin and Eli Glasner.
\newblock Residual properties and almost equicontinuity.
\newblock {\em J. Anal. Math.}, 84:243--286, 2001.

\bibitem[Aus88]{Au88}
Joseph Auslander.
\newblock {\em Minimal flows and their extensions}, volume 153 of {\em
  North-Holland Math. Stud.}
\newblock Amsterdam etc.: North-Holland, 1988.

\bibitem[BS75]{BS75}
Walter Bauer and Karl Sigmund.
\newblock Topological dynamics of transformations induced on the space of
  probability measures.
\newblock {\em Monatsh. Math.}, 79:81--92, 1975.

\bibitem[DE72]{DE72}
Seymour~Z. Ditor and Larry Eifler.
\newblock Some open mapping theorems for measures.
\newblock {\em Trans. Am. Math. Soc.}, 164:287--293, 1972.

\bibitem[DGS76]{DGS76}
Manfred Denker, Christian Grillenberger, and Karl Sigmund.
\newblock {\em Ergodic theory on compact spaces}, volume Vol. 527 of {\em
  Lecture Notes in Mathematics}.
\newblock Springer-Verlag, Berlin-New York, 1976.

\bibitem[Dow05]{Dow05}
Tomasz Downarowicz.
\newblock Survey of odometers and {T}oeplitz flows.
\newblock In {\em Algebraic and topological dynamics}, volume 385 of {\em
  Contemp. Math.}, pages 7--37. Amer. Math. Soc., Providence, RI, 2005.

\bibitem[DX24]{DX24}
Xiongping Dai and Yuxuan Xie.
\newblock Characterizations of open and semi-open maps of compact {Hausdorff}
  spaces by induced maps.
\newblock {\em Topology Appl.}, 350:10, 2024.
\newblock Id/No 108921.

\bibitem[Ell69]{ellis1969lectures}
Robert Ellis.
\newblock Lectures on topological dynamics.
\newblock New {York}: {W}. {A}. {Benjamin}, {Inc}., xv, 211 p. (1969)., 1969.

\bibitem[FW89]{FW89}
Hillel Furstenberg and Benjamin Weiss.
\newblock On almost {$1$}-{$1$} extensions.
\newblock {\em Israel J. Math.}, 65(3):311--322, 1989.

\bibitem[Gla00]{Gla00}
Eli Glasner.
\newblock Quasifactors of minimal systems.
\newblock {\em Topol. Methods Nonlinear Anal.}, 16(2):351--370, 2000.

\bibitem[Gla07]{Gl07}
Eli Glasner.
\newblock The structure of tame minimal dynamical systems.
\newblock {\em Ergodic Theory Dyn. Syst.}, 27(6):1819--1837, 2007.

\bibitem[GW95]{GW95}
Eli Glasner and Benjamin Weiss.
\newblock Quasi-factors of zero-entropy systems.
\newblock {\em J. Amer. Math. Soc.}, 8(3):665--686, 1995.

\bibitem[Hos97]{Hos97}
Hiroshi Hosokawa.
\newblock Induced mappings on hyperspaces.
\newblock {\em Tsukuba J. Math.}, 21(1):239--250, 1997.

\bibitem[KO07]{KO07}
Dominik Kwietniak and Piotr Oprocha.
\newblock Topological entropy and chaos for maps induced on hyperspaces.
\newblock {\em Chaos Solitons Fractals}, 33(1):76--86, 2007.

\bibitem[Mic51]{E51}
Ernest Michael.
\newblock Topologies on spaces of subsets.
\newblock {\em Trans. Amer. Math. Soc.}, 71:152--182, 1951.

\bibitem[Nag22]{Nag22}
Anima Nagar.
\newblock Characterization of quasifactors.
\newblock {\em J. Fixed Point Theory Appl.}, 24(1):Paper No. 12, 17, 2022.

\bibitem[Phe01]{Phe01}
Robert~R. Phelps.
\newblock {\em Lectures on {Choquet}'s theorem}, volume 1757 of {\em Lect.
  Notes Math.}
\newblock Berlin: Springer, 2001.

\bibitem[Wei12]{W12}
Benjamin Weiss.
\newblock Minimal models for free actions.
\newblock In {\em Dynamical systems and group actions}, volume 567 of {\em
  Contemp. Math.}, pages 249--264. Amer. Math. Soc., Providence, RI, 2012.

\end{thebibliography}

\end{document}